\documentclass[10pt]{amsart}%
\usepackage{amssymb,amsmath,amsfonts,latexsym,amsthm,geometry,graphicx}
\usepackage{amsmath}
\usepackage{amsfonts}
\usepackage{color}
\usepackage{amssymb}
\usepackage{graphicx}%
\providecommand{\U}[1]{\protect\rule{.1in}{.1in}}
\providecommand{\U}[1]{\protect\rule{.1in}{.1in}}
\providecommand{\U}[1]{\protect\rule{.1in}{.1in}}
\newtheorem{theorem}{Theorem}[section]
\newtheorem{corollary}[theorem]{Corollary}

\newtheorem{lemma}[theorem]{Lemma}

\theoremstyle{definition}

\newtheorem{remark}[theorem]{Remark}

\begin{document}
\title{On the constants of the Bohnenblust-Hille and Hardy--Littlewood inequalities}
\author[G. Ara\'{u}jo]{Gustavo Ara\'{u}jo}
\address{Departamento de Matem\'{a}tica \\
Universidade Federal da Para\'{\i}ba \\
58.051-900 - Jo\~{a}o Pessoa, Brazil.}
\email{gdasaraujo@gmail.com}
\author[D. Pellegrino]{Daniel Pellegrino}
\address{Departamento de Matem\'{a}tica \\
Universidade Federal da Para\'{\i}ba \\
58.051-900 - Jo\~{a}o Pessoa, Brazil.}
\email{pellegrino@pq.cnpq.br and dmpellegrino@gmail.com}
\thanks{The authors are supported by CNPq Grant 313797/2013-7 - PVE - Linha 2}

\begin{abstract}
In this paper, among other results, we improve the best known estimates for
the constants of the generalized Bohnenblust-Hille inequality. These
enhancements are then used to improve the best known constants of the
Hardy--Littlewood inequality; this inequality asserts that for a positive integer $m\geq2$ with $2m\leq
p\leq\infty$ and $\mathbb{K}=\mathbb{R}$ or $\mathbb{C}$ there exists a
constant $C_{m,p}^{\mathbb{K}}\geq1$ such that, for all continuous $m$--linear
forms $T:\ell_{p}^{n}\times\cdots\times\ell_{p}^{n}\rightarrow\mathbb{K}$, and
all positive integers $n$,%
\[
\left(  \sum_{j_{1},...,j_{m}=1}^{n}\left\vert T(e_{j_{1}},...,e_{j_{m}%
})\right\vert ^{\frac{2mp}{mp+p-2m}}\right)  ^{\frac{mp+p-2m}{2mp}}\leq
C_{m,p}^{\mathbb{K}}\left\Vert T\right\Vert ,
\]
and the exponent $\frac{2mp}{mp+p-2m}$ is sharp. In particular, we show that
for $p > 2m^{3}-4m^{2}+2m$ the optimal constants satisfying the above
inequality are dominated by the best known estimates for the constants of the
$m$-linear Bohnenblust--Hille inequality. More precisely if $\gamma$ denotes
the Euler--Mascheroni constant, considering the case of complex scalars as an
illustration, we show that%
\[
C_{m,p}^{\mathbb{C}}\leq\prod\limits_{j=2}^{m}\Gamma\left(  2-\frac{1}%
{j}\right)  ^{\frac{j}{2-2j}}<m^{\frac{1-\gamma}{2}},
\]
which is somewhat surprising since this new formula has no dependence on $p$
(the former estimate depends on $p$ but, paradoxally, is worse than this new
one). This suggest the following open problems:

1) Are the optimal constants of the Hardy--Littlewood inequality and
Bohnenblust--Hille inequalities the same?

2) Are the optimal constants of the Hardy--Littlewood inequality independent of
$p$ (at least for large $p$)?

\end{abstract}
\maketitle

\section{Introduction}

\label{sec1}

Let $\mathbb{K}$ be $\mathbb{R}$ or $\mathbb{C}$ and $m\geq2$ be a positive
integer. In 1931, F. Bohnenblust and E. Hille (see \cite{bh}) proved in the \textit{Annals of Mathematics} that
there exists a constant $B_{\mathbb{K},m}^{\mathrm{mult}}\geq1$ such that for
all continuous $m$--linear forms $T:\ell_{\infty}^{n}\times\cdots\times
\ell_{\infty}^{n}\rightarrow\mathbb{K}$, and all positive integers $n$,%
\begin{equation}
\left(  \sum_{j_{1},...,j_{m}=1}^{n}\left\vert T(e_{j_{1}},...,e_{j_{m}%
})\right\vert ^{\frac{2m}{m+1}}\right)  ^{\frac{m+1}{2m}}\leq B_{\mathbb{K}%
,m}^{\mathrm{mult}}\left\Vert T\right\Vert . \label{tttt}%
\end{equation}
The task of estimating the constants $B_{\mathbb{K},m}^{\mathrm{mult}}$ of
this inequality (now known as the Bohnenblust--Hille inequality) is nowadays a
challenging problem in Mathematical Analysis. For complex scalars, having good
estimates for the polynomial version of $B_{\mathbb{K},m}^{\mathrm{mult}}$ is
crucial in applications in Complex Analysis and Analytic Number Theory (see
\cite{annalss}); for real scalars, the estimates of $B_{\mathbb{R}%
,m}^{\mathrm{mult}}$ are important in Quantum Information Theory (see
\cite{monta}). In the last years a series of papers related to the
Bohnenblust--Hille inequality have been published and several advances were
achieved (see \cite{alb,da, annalss,dps, Nuuu, psss,q} and the references
therein). For instance, the subexponentiality of the constants of the
polynomial version of the Bohnenblust--Hille inequality (case of complex
scalars) was recenly used in \cite{bohr} to obtain the asymptotic growth of
the Bohr radius of the $n$-dimensional polydisk. More precisely, according to
Boas and Khavinson \cite{BK97}, the Bohr radius $K_{n}$ of the $n$-dimensional
polydisk is the largest positive number $r$ such that all polynomials
$\sum_{\alpha}a_{\alpha}z^{\alpha}$ on $\mathbb{C}^{n}$ satisfy
\[
\sup_{z\in r\mathbb{D}^{n}}\sum_{\alpha}|a_{\alpha}z^{\alpha}|\leq\sup
_{z\in\mathbb{D}^{n}}\left\vert \sum_{\alpha}a_{\alpha}z^{\alpha}\right\vert
.
\]
The Bohr radius $K_{1}$ was estimated by H. Bohr, and it was later shown
independently by M. Riesz, I. Schur and F. Wiener that $K_{1}=1/3$ (see
\cite{BK97, bbohr} and the references therein). For $n\geq2$, exact values of
$K_{n}$ are unknown. In \cite{bohr}, the subexponentiality of the constants of
the complex polynomial version of the Bohnenblust--Hille inequality was proved
and using this fact it was finally proved that%
\[
\lim_{n\rightarrow\infty}\frac{K_{n}}{\sqrt{\frac{\log n}{n}}}=1,
\]
solving a challenging problem that many researchers have been chipping away at
for several years.

The best known estimates for the constants in (\ref{tttt}), which are recently
presented in \cite{bohr}, are
\begin{equation}%
\begin{array}
[c]{llll}%
\mathrm{B}_{\mathbb{C},m}^{\mathrm{mult}} & \leq & \displaystyle\prod
\limits_{j=2}^{m}\Gamma\left(  2-\frac{1}{j}\right)  ^{\frac{j}{2-2j}}, &
\vspace{0.2cm}\\
\mathrm{B}_{\mathbb{R},m}^{\mathrm{mult}} & \leq & 2^{\frac{446381}%
{55440}-\frac{m}{2}}\displaystyle\prod\limits_{j=14}^{m}\left(  \frac
{\Gamma\left(  \frac{3}{2}-\frac{1}{j}\right)  }{\sqrt{\pi}}\right)
^{\frac{j}{2-2j}}, & \text{ for }m\geq14,\vspace{0.2cm}\\
\mathrm{B}_{\mathbb{R},m}^{\mathrm{mult}} & \leq & \displaystyle\prod
\limits_{j=2}^{m}2^{\frac{1}{2j-2}}, & \text{ for }2\leq m\leq13.
\end{array}
\label{khfr}%
\end{equation}

In a more friendly presentation the above formulas tell us that the growth of
the constants $\mathrm{B}_{\mathbb{K},m}^{\mathrm{mult}}$ is subpolynomial (in
fact, sublinear) since, from the above estimates it can be proved that (see
\cite{bohr})
\[%
\begin{array}
[c]{rcl}%
\mathrm{B}_{\mathbb{C},m}^{\mathrm{mult}} & < & \displaystyle m^{\frac
{1-\gamma}{2}}<m^{0.21139},\vspace{0.2cm}\\
\mathrm{B}_{\mathbb{R},m}^{\mathrm{mult}} & < & 1.3\cdot m^{\frac
{2-\log2-\gamma}{2}}<1.3\cdot m^{0.36482},
\end{array}
\]
where $\gamma$ denotes the Euler--Mascheroni constant. This is a quite
surprising result since the previous estimates (from 1931-2011) predicted an
exponential growth; it was only in 2012, with \cite{psss}, motivated by
\cite{dps}, that the panorama started to change.

The Hardy--Littlewood inequality is a natural extension of the
Bohnenblust--Hille inequality (see \cite{alb,hardy,pra}) and asserts that for
$4\leq2m\leq p\leq\infty$ there exists a constant $C_{m,p}^{\mathbb{K}}\geq1$ such
that, for all continuous $m$--linear forms $T:\ell_{p}^{n}\times\cdots
\times\ell_{p}^{n}\rightarrow\mathbb{K}$, and all positive integers $n$,
\[
\left(  \sum_{j_{1},...,j_{m}=1}^{n}\left\vert T(e_{j_{1}},...,e_{j_{m}%
})\right\vert ^{\frac{2mp}{mp+p-2m}}\right)  ^{\frac{mp+p-2m}{2mp}}\leq
C_{m,p}^{\mathbb{K}}\left\Vert T\right\Vert ,
\]
and the exponent $\frac{2mp}{mp+p-2m}$ is optimal.

\begin{figure}[h]
\centering
\includegraphics[scale=0.8]{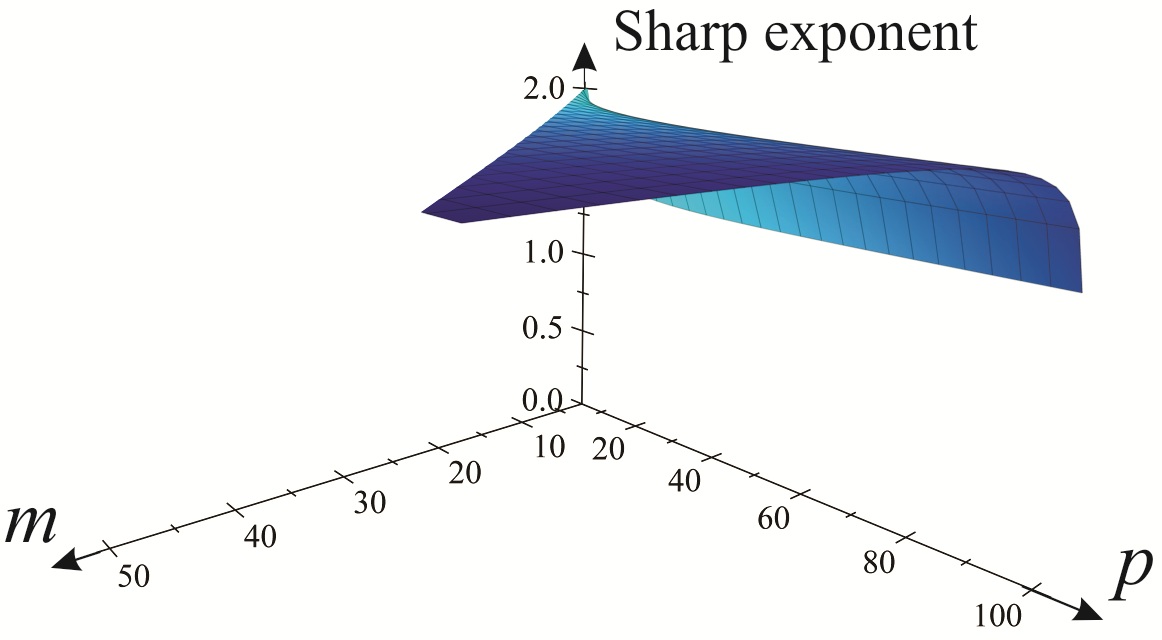}
\end{figure}

The original estimates for
$C_{m,p}^{\mathbb{K}}$ were of the form
\begin{equation}
C_{m,p}^{\mathbb{K}}\leq\left(  \sqrt{2}\right)  ^{m-1} \label{8u7}%
\end{equation}
(see \cite{alb}); nowadays the best known estimates for the constants
$C_{m,p}^{\mathbb{K}}$ (see \cite{ap}) are
\begin{equation}%
\begin{array}
[c]{llll}%
C_{m,p}^{\mathbb{C}} & \leq\left(  \frac{2}{\sqrt{\pi}}\right)  ^{\frac
{2m(m-1)}{p}}\left(  \displaystyle\prod\limits_{j=2}^{m}\Gamma\left(
2-\frac{1}{j}\right)  ^{\frac{j}{2-2j}}\right)  ^{\frac{p-2m}{p}}, &
\vspace{0.2cm} & \\
C_{m,p}^{\mathbb{R}} & \leq\left(  \sqrt{2}\right)  ^{\frac{2m\left(
m-1\right)  }{p}}\left(  2^{\frac{446381}{55440}-\frac{m}{2}}%
\displaystyle\prod\limits_{j=14}^{m}\left(  \frac{\Gamma\left(  \frac{3}%
{2}-\frac{1}{j}\right)  }{\sqrt{\pi}}\right)  ^{\frac{j}{2-2j}}\right)
^{\frac{p-2m}{p}}, & \text{ for }m\geq14,\vspace{0.2cm} & \\
C_{m,p}^{\mathbb{R}} & \leq\left(  \sqrt{2}\right)  ^{\frac{2m\left(
m-1\right)  }{p}}\left(  \displaystyle\prod\limits_{j=2}^{m}2^{\frac{1}{2j-2}%
}\right)  ^{\frac{p-2m}{p}}, & \text{ for }2\leq m\leq13. &
\end{array}
\label{btt}%
\end{equation}
Note that the presence of the parameter $p$ in the formulas of (\ref{btt}), if
compared to (\ref{8u7}), catches more subtle information since now it is clear
that the estimates become \textquotedblleft better\textquotedblright\ as $p$
grows. As $p$ tends to infinity we note that the above estimates tend to best
known estimates for $B_{\mathbb{K},m}^{\mathrm{mult}}$ (see \eqref{khfr}). In
this paper, among other results, we show that for $p > 2m^{3}-4m^{2}+2m$ the
constant $C_{m,p}^{\mathbb{K}}$ has the exactly same upper bounds that we have
for the Bohnenblust--Hille constants \eqref{khfr}. More precisely we shall
show that if $p > 2m^{3}-4m^{2}+2m$, then%
\begin{equation}%
\begin{array}
[c]{llll}%
C_{m,p}^{\mathbb{C}} & \leq & \displaystyle\prod\limits_{j=2}^{m}\Gamma\left(
2-\frac{1}{j}\right)  ^{\frac{j}{2-2j}}, & \vspace{0.2cm}\\
C_{m,p}^{\mathbb{R}} & \leq & 2^{\frac{446381}{55440}-\frac{m}{2}%
}\displaystyle\prod\limits_{j=14}^{m}\left(  \frac{\Gamma\left(  \frac{3}%
{2}-\frac{1}{j}\right)  }{\sqrt{\pi}}\right)  ^{\frac{j}{2-2j}}, & \text{ for
}m\geq14,\vspace{0.2cm}\\
C_{m,p}^{\mathbb{R}} & \leq & \displaystyle\prod\limits_{j=2}^{m}2^{\frac
{1}{2j-2}}, & \text{ for }2\leq m\leq13.
\end{array}
\label{9ww}%
\end{equation}

It is not difficult to verify that (\ref{9ww}) in fact improves (\ref{btt}).
However the most interesting point is that in (\ref{9ww}), contrary to
(\ref{btt}), we have no dependence on $p$ in the formulas and, besides, these
new estimates are precisely the best known estimates for the constants of the
Bohnenblust--Hille inequality (see \eqref{khfr}).

To prove these new estimates we also improve the best known estimates for the
generalized Bohnenblust--Hille inequality. Recall that the generalized
Bohnenblust--Hille inequality (see \cite{alb}) asserts that if $\left(
q_{1},\ldots,q_{m}\right)  \in\lbrack1,2]^{m}$ are so that%
\[
\frac{1}{q_{1}}+\cdots+\frac{1}{q_{m}}=\frac{m+1}{2},
\]
then there is $B_{m,\left(  q_{1},...,q_{m}\right)  }^{\mathbb{K}}\geq1$ such
that%
\[
\left(  \sum_{j_{1}=1}^{n}\dots\left(  \sum_{j_{m}=1}^{n}\left\vert T\left(
e_{j_{1}},...,e_{j_{m}}\right)  \right\vert ^{q_{m}}\right)  ^{\frac{q_{m-1}%
}{q_{m}}}\dots\right)  ^{\frac{1}{q_{1}}}\leq B_{m,\left(  q_{1}%
,...,q_{m}\right)  }^{\mathbb{K}}\Vert T\Vert
\]
for all $m$-linear forms $T:\ell_{\infty}^{n}\times\cdots\times\ell_{\infty
}^{n}\rightarrow\mathbb{K}$, and all positive integers $n.$ The importance of
this result trancends the intrinsic mathematical novelty since, as it was
recently shown (see \cite{bohr}), this new approach is fundamental to improve
the estimates of the constants of the classical Bohnenblust--Hille inequality.
The best known estimates for the constants $B_{m,\left(  q_{1},...,q_{m}%
\right)  }^{\mathbb{K}}$ are presented in \cite{n}. More precisely, for
complex scalars and $1\leq q_{1}\leq\cdots\leq q_{m}\leq2$, from \cite{n} we
know that, for $\mathbf{q}=\left(  q_{1},...,q_{m}\right)  $,%

\[
B_{m,\mathbf{q} }^{\mathbb{C}}\leq\left(  \displaystyle\prod\limits_{j=1}%
^{m}\Gamma\left(  2-\frac{1}{j}\right)  ^{\frac{j}{2-2j}}\right)  ^{2m\left(
\frac{1}{q_{m}}-\frac{1}{2}\right)  }\left(  \prod_{k=1}^{m-1}\left(
\Gamma\left(  \frac{3k+1}{2k+2}\right)  ^{\left(  \frac{-k-1}{2k}\right)
\left(  m-k\right)  }\displaystyle\prod\limits_{j=1}^{k}\Gamma\left(
2-\frac{1}{j}\right)  ^{\frac{j}{2-2j}}\right)  ^{2k\left(  \frac{1}{q_{k}%
}-\frac{1}{q_{k+1}}\right)  }\right)  .
\]

In the present paper we improve the above estimates for a certain family of
$\left(  q_{1},...,q_{m}\right)  $. More precisely, if
\[
\max q_{i}<\frac{2m^{2}-4m+2}{m^{2}-m-1},
\]
then%
\[
B_{m,\left(  q_{1},...,q_{m}\right)  }^{\mathbb{C}}\leq\displaystyle\prod
\limits_{j=2}^{m}\Gamma\left(  2-\frac{1}{j}\right)  ^{\frac{j}{2-2j}}.
\]
A similar result holds for real scalars. These results have a crucial
importance in the next sections.

The paper is organized as follows. In Section \ref{sec2} we obtain new
estimates for the generalized Bohnenblust--Hille inequality and in Section
\ref{sec3} we use these estimates to prove new estimates for the constants of
the Hardy-Littlewood inequality. In the final section (Section \ref{sec4}) the
estimates of the previous sections are used to obtain new constants for the
generalized Hardy--Littlewood inequality.

\section{New estimates for the constants of the generalized Bohnenblust--Hille inequality}

\label{sec2}

We recall that the Khinchine inequality (see \cite{Di}) asserts that for any
$0<q<\infty$, there are positive constants $A_{q}$, $B_{q}$ such that
regardless of the scalar sequence $(a_{j})_{j=1}^{n}$ we have
\[
A_{q}\left(  \sum_{j=1}^{n}|a_{j}|^{2}\right)  ^{\frac{1}{2}}\leq\left(
\int_{0}^{1}\left\vert \sum_{j=1}^{n}a_{j}r_{j}(t)\right\vert ^{q}dt\right)
^{\frac{1}{q}}\leq B_{q}\left(  \sum_{j=1}^{n}|a_{j}|^{2}\right)  ^{\frac
{1}{2}},
\]
where $r_{j}$ are the Rademacher functions. More generally, from the above
inequality together with the Minkowski inequality we know that (see \cite{ap},
for instance, and the references therein)
\begin{equation}
A_{q}^{m}\left(  \sum_{j_{1},...,j_{m}=1}^{n}|a_{j_{1}...j_{m}}|^{2}\right)
^{\frac{1}{2}}\leq\left(  \int_{I}\left\vert \sum_{j_{1},...,j_{m}=1}%
^{n}a_{j_{1}...j_{m}}r_{j_{1}}(t_{1})...r_{j_{m}}(t_{m})\right\vert
^{q}dt\right)  ^{\frac{1}{q}}\leq B_{q}^{m}\left(  \sum_{j_{1},...,j_{m}%
=1}^{n}|a_{j_{1}...j_{m}}|^{2}\right)  ^{\frac{1}{2}}, \label{mmjj}%
\end{equation}
where $I=[0,1]^{m}$ and $dt=dt_{1}...dt_{m}$, for all scalar sequences
$\left(  a_{j_{1}....j_{m}}\right)  _{j_{1},...,j_{m}=1}^{n}$.

The best constants $A_{q}$ are known (see \cite{haage}). Indeed,

\begin{itemize}
\item $\displaystyle A_{q}=\sqrt{2}\left(  \frac{\Gamma\left(
\frac{1+q}{2}\right)  }{\sqrt{\pi}}\right)  ^{\frac{1}{q}}$ if $q>q_{0}\cong1.8474$;
\vspace{0.2cm}

\item $\displaystyle A_{q}=2^{\frac{1}{2}-\frac{1}{q}}$ if $q<q_{0}$.
\end{itemize}

The definition of the number $q_{0}$ above is the following: $q_{0}\in(1,2)$
is the unique real number with
\[
\Gamma\left(  \frac{p_{0}+1}{2}\right)  =\frac{\sqrt{\pi}}{2}.
\]
For complex scalars, using Steinhaus variables instead of Rademacher functions
it is well known that a similar inequality holds, but with better
constants\ (see \cite{konig, sawa}). In this case the optimal constant is

\begin{itemize}
\item $\displaystyle A_{q}=\Gamma\left(  \frac{q+2}{2}\right)  ^{\frac{1}{q}}$ if
$q\in\lbrack1,2]$.
\end{itemize}

The notation of the constant $A_{q}$ above will be used in all this paper.

\begin{lemma}
\label{81} Let $\left(  q_{1},\ldots,q_{m}\right)  \in\lbrack1,2]^{m}$ such
that $\frac{1}{q_{1}}+\cdots+\frac{1}{q_{m}}=\frac{m+1}{2}$. If $q_{i}%
\geq\frac{2m-2}{m}$ for some index $i$ and $q_{k}=q_{l}$ for all $k\neq i$ and
$l\neq i$, then
\[
B_{m,\left(  q_{1},...,q_{m}\right)  }^{\mathbb{K}}\leq\prod\limits_{j=2}%
^{m}A_{\frac{2j-2}{j}}^{-1},
\]
where $A_{\frac{2j-2}{j}}$ are the respective constants of the Khnichine inequality.
\end{lemma}

\begin{proof}
There is no loss of generality in supposing that $i=1$. By using the multiple
Khinchine inequality (\ref{mmjj}) we have (see \cite[Section 2]{ap} for
details)
\[
\left(  \sum\limits_{j_{1},...,j_{m-1}=1}^{n}\left(  \sum\limits_{j_{m}=1}%
^{n}\left\vert T\left(  e_{j_{1}},...,e_{j_{m}}\right)  \right\vert
^{2}\right)  ^{\frac{1}{2}\frac{2m-2}{m}}\right)  ^{\frac{m}{2m-2}}\leq
A_{\frac{2m-2}{m}}^{-1}B_{\mathbb{K},m-1}^{\mathrm{mult}}\left\Vert
T\right\Vert .
\]
From \cite{bohr} we know that
\[
B_{\mathbb{K},1}^{\mathrm{mult}}=1\qquad\text{ and }\qquad B_{\mathbb{K}%
,m}^{\mathrm{mult}}\leq A_{\frac{2m-2}{m}}^{-1}B_{\mathbb{K},m-1}%
^{\mathrm{mult}},
\]
and thus
\[
A_{\frac{2m-2}{m}}^{-1}B_{\mathbb{K},m-1}^{\mathrm{mult}}=\prod\limits_{j=2}%
^{m}A_{\frac{2j-2}{j}}^{-1}.%
\]

Taking the $m$ exponents%
\[%
\begin{array}
[c]{c}%
\displaystyle\left(  \frac{2m-2}{m},...,\frac{2m-2}{m},2\right)  ,\\
\vdots\\
\displaystyle\left(  2,\frac{2m-2}{m},....,\frac{2m-2}{m}\right)  ,
\end{array}
\]
interpolated (in the sense of \cite{alb}) with $\theta_{1}=\cdots=\theta
_{m-1}=\frac{2}{q_{1}}-1$ and $\theta_{m}=m-\frac{2m-2}{q_{1}},$ we conclude
that the exponent obtained is $\left(  q_{1},...,q_{m}\right)  $. Since
$\frac{2m-2}{m}<2,$ from a repeated use of the Minkowski inequality (in the lines of the arguments from \cite{alb}) we know that the constants associated to all the above exponents
are dominated by $\prod\limits_{j=2}^{m}A_{\frac{2j-2}{j}}^{-1},$ and the
proof is done.
\end{proof}

From now on, for any function $f$, whenever it makes sense we formally define
$f(\infty)=\lim_{p\to\infty}f(p)$.

\begin{lemma}
\label{lema} Let $m\geq2$ be a positive integer, let $2m<p\leq\infty$, let
$q_{1},...,q_{m}\in\left[  \frac{p}{p-m},2\right]  $. If
\begin{equation}
\frac{1}{q_{1}}+\cdots+\frac{1}{q_{m}}=\frac{mp+p-2m}{2p}, \label{3121}%
\end{equation}
then, for all $s\in\left(  \max q_{i},2\right]  $, the vector $\left(
q_{1}^{-1},...,q_{m}^{-1}\right)  $ belongs to the convex hull in
$\mathbb{R}^{m}$ of
\[
\left\{  \sum\limits_{k=1}^{m}a_{1k}e_{k},...,\sum\limits_{k=1}^{m}a_{mk}%
e_{k}\right\}  ,
\]
where
\[
a_{jk}=\left\{
\begin{array}
[c]{ll}%
s^{-1}, & \text{ if }k\neq j\\
\lambda_{m,s}^{-1}, & \text{ if }k=j
\end{array}
\right.
\]
and
\[
\lambda_{m,s}=\frac{2ps}{mps+ps+2p-2mp-2ms}.
\]
Equivalently, we say that the exponent $\left(  q_{1},...,q_{m}\right)  $ is
the interpolation of the $m$ exponents $\left(  s,...,s,\lambda_{m,s}\right),$ ...,$\left(  \lambda_{m,s},s,...,s\right)$.
\end{lemma}

\begin{proof}
We want to prove that for $\left(  q_{1},...,q_{m}\right)  \in\left[  \frac
{p}{p-m},2\right]  ^{m}$ and $s\in\left(  \max q_{i},2\right]  $ there are
$0<\theta_{j,s}<1$, $j=1,...,m$, such that
\[%
\begin{array}
[c]{c}%
\displaystyle\sum_{j=1}^{m}\theta_{j,s}=1,\vspace{0.2cm}\\
\displaystyle\frac{1}{q_{1}}=\frac{\theta_{1,s}}{\lambda_{m,s}}+\frac
{\theta_{2,s}}{s}+\cdots+\frac{\theta_{m,s}}{s},\vspace{0.2cm}\\
\displaystyle\vdots\vspace{0.2cm}\\
\displaystyle\frac{1}{q_{m}}=\frac{\theta_{1,s}}{s}+\cdots+\frac
{\theta_{m-1,s}}{s}+\frac{\theta_{m,s}}{\lambda_{m,s}}.
\end{array}
\]

Observe initially that from \eqref{3121} we have
\[
\max q_{i}\geq\frac{2mp}{mp+p-2m}.
\]
Note also that for all $s\in\left[\frac{2mp-2p}{mp-2m},2\right]$ we have
\begin{equation}\label{invt1}
mps+ps+2p-2mp-2ms>0 \qquad \text{and} \qquad \frac{p}{p-m}\leq \lambda_{m,s}\leq 2.
\end{equation}

Since $s>\max q_{i}\geq\frac{2mp}{mp+p-2m}>\frac{2mp-2p}{mp-2m}$ (the last inequality is strict because we are not considering the case $p=2m$) it follows that $\lambda_{m,s}$ is well defined for all $s\in\left(\max q_i,2\right]$. Furthermore, for all $s>\frac{2mp}{mp+p-2m}$ it is possible to prove that $\lambda_{m,s}<s$. In fact, $s>\frac{2mp}{mp+p-2m}$ implies $mps+ps-2ms>2mp$ and thus adding $2p$ in both sides of this inequality we can conclude that $$\frac{2ps}{mps+ps+2p-2mp-2ms}<\frac{2ps}{2p}=s,$$ i.e.,
\begin{equation}
\label{lambdams}
\lambda_{m,s}<s.
\end{equation}

For each $j=1,...,m$, consider
\[
\theta_{j,s}=\frac{\lambda_{m,s}\left(  s-q_{j}\right)  }{q_{j}\left(
s-\lambda_{m,s}\right)  }.
\]

Since $\sum_{j=1}^{m}\frac{1}{q_{j}}=\frac{mp+p-2m}{2p}$ we conclude that
\begin{align*}
\sum_{j=1}^{m}\theta_{j,s}  &  =\sum_{j=1}^{m}\frac{\lambda_{m,s}\left(
s-q_{j}\right)  }{q_{j}\left(  s-\lambda_{m,s}\right)  }\\
&  =\frac{\lambda_{m,s}}{s-\lambda_{m,s}}\left(  s\sum_{j=1}^{m}\frac{1}%
{q_{j}}-m\right) \\
&  =1.
\end{align*}
Since by hypothesis $s>\max q_{i}\geq q_{j}$ for all $j=1,...,m$, it follows
that $\theta_{j,s}>0$ for all $j=1,...,m$ and thus
\[
0<\theta_{j,s}<\sum_{j=1}^{m}\theta_{j,s}=1.
\]

Finally, note that
\[
\frac{\theta_{j,s}}{\lambda_{m,s}}+\frac{1-\theta_{j,s}}{s}=\frac
{\frac{\lambda_{m,s}\left(  s-q_{j}\right)  }{q_{j}\left(  s-\lambda
_{m,s}\right)  }}{\lambda_{m,s}}+\frac{1-\frac{\lambda_{m,s}\left(
s-q_{j}\right)  }{q_{j}\left(  s-\lambda_{m,s}\right)  }}{s}=\allowbreak
\frac{1}{q_{j}}.
\]
Therefore
\[%
\begin{array}
[c]{c}%
\displaystyle\frac{1}{q_{1}}=\frac{\theta_{1,s}}{\lambda_{m,s}}+\frac
{\theta_{2,s}}{s}+\cdots+\frac{\theta_{m,s}}{s},\vspace{0.2cm}\\
\displaystyle\vdots\vspace{0.2cm}\\
\displaystyle\frac{1}{q_{m}}=\frac{\theta_{1,s}}{s}+\cdots+\frac
{\theta_{m-1,s}}{s}+\frac{\theta_{m,s}}{\lambda_{m,s}},
\end{array}
\]
and the proof is done.
\end{proof}

Combining the two previous lemmata we have:

\begin{theorem}
\label{zmuzmu}
Let $m\geq2$ be a positive integer and $q_{1},...,q_{m}\in\left[  1,2\right]
$. If
\[
\frac{1}{q_{1}}+\cdots+\frac{1}{q_{m}}=\frac{m+1}{2},
\]
and
\[
\max q_{i}<\frac{2m^{2}-4m+2}{m^{2}-m-1},
\]
then
\[
B_{m,\left(  q_{1},...,q_{m}\right)  }^{\mathbb{K}}\leq\prod\limits_{j=2}%
^{m}A_{\frac{2j-2}{j}}^{-1},
\]
where $A_{\frac{2j-2}{j}}$ are the respective constants of the Khnichine inequality.
\end{theorem}

\begin{proof}
Let
\[
s=\frac{2m^{2}-4m+2}{m^{2}-m-1}\qquad\text{ and }\qquad q=\frac{2m-2}{m}.
\]
Since
\[
\frac{m-1}{s}+\frac{1}{q}=\frac{m+1}{2},
\]
from Lemma \ref{81} the Bohnenblust--Hille exponents
\[
\left(  t_{1},...,t_{m}\right)  =\left(  s,...,s,q\right)  ,...,\left(
q,s,...,s\right)
\]
are associated to
\[
B_{m,\left(  t_{1},...,t_{m}\right)  }^{\mathbb{K}}\leq\prod\limits_{j=2}%
^{m}A_{\frac{2j-2}{j}}^{-1}.
\]
Since by hypothesis
\[
\max q_{i}<\frac{2m^{2}-4m+2}{m^{2}-m-1}=s,
\]
from the previous lemma (Lemma \ref{lema}) with $p=\infty$, the exponent
$\left(  q_{1},...,q_{m}\right)  $ is the interpolation of
\[
\left(  \frac{2s}{ms+s+2-2m},s,...,s\right)  ,...,\left(  s,...,s,\frac
{2s}{ms+s+2-2m}\right)  .
\]
But note that%
\[
\frac{2s}{ms+s+2-2m}=\frac{2m-2}{m}%
\]
and from Lemma \ref{81} they are associated to the constants
\[
B_{m,\left(  q_{1},...,q_{m}\right)  }^{\mathbb{K}}\leq\prod\limits_{j=2}%
^{m}A_{\frac{2j-2}{j}}^{-1}.
\]
\end{proof}

\begin{corollary}
Let $m\geq2$ be a positive integer and $q_{1},...,q_{m}\in\left[  1,2\right]
$. If
\[
\frac{1}{q_{1}}+\cdots+\frac{1}{q_{m}}=\frac{m+1}{2},
\]
and
\[
\max q_{i}<\frac{2m^{2}-4m+2}{m^{2}-m-1},
\]
then%
\[%
\begin{array}
[c]{llll}%
B_{m,\left(  q_{1},...,q_{m}\right)  }^{\mathbb{C}} & \leq &
\displaystyle\prod\limits_{j=2}^{m}\Gamma\left(  2-\frac{1}{j}\right)
^{\frac{j}{2-2j}}, & \vspace{0.2cm}\\
B_{m,\left(  q_{1},...,q_{m}\right)  }^{\mathbb{R}} & \leq & 2^{\frac
{446381}{55440}-\frac{m}{2}}\displaystyle\prod\limits_{j=14}^{m}\left(
\frac{\Gamma\left(  \frac{3}{2}-\frac{1}{j}\right)  }{\sqrt{\pi}}\right)
^{\frac{j}{2-2j}}, & \text{ for }m\geq14,\vspace{0.2cm}\\
B_{m,\left(  q_{1},...,q_{m}\right)  }^{\mathbb{R}} & \leq &
\displaystyle\prod\limits_{j=2}^{m}2^{\frac{1}{2j-2}}, & \text{ for }2\leq
m\leq13.
\end{array}
\]

\end{corollary}

\section{Application 1: Improving the constants of the Hardy-Littlewood inequality}

\label{sec3}

The main result of this section shows that for $p > 2m^{3}-4m^{2}+2m$ the
optimal constants satisfying the Hardy--Littlewood inequality for $m$-linear
forms in $\ell_{p}$ spaces are dominated by the best known estimates for the
constants of the $m$-linear Bohnenblust--Hille inequality; this result
improves the recent estimates (see \eqref{btt}), and may suggest a more subtle
connection between the optimal constants of these inequalities.

\begin{theorem}
\label{667}Let $m\geq2$ be a positive integer and $2m^{3}-4m^{2}+2m < p\leq\infty.$ Then, for all continuous $m$--linear forms $T:\ell_{p}^{n}%
\times\cdots\times\ell_{p}^{n}\rightarrow\mathbb{K}$ and all positive integers
$n$, we have%
\begin{equation}
\left(  \sum_{j_{1},...,j_{m}=1}^{n}\left\vert T(e_{j_{1}},...,e_{j_{m}%
})\right\vert ^{\frac{2mp}{mp+p-2m}}\right)  ^{\frac{mp+p-2m}{2mp}}\leq\left(
\prod\limits_{j=2}^{m}A_{\frac{2j-2}{j}}^{-1}\right)  \left\Vert T\right\Vert
. \label{665}%
\end{equation}

\end{theorem}

\begin{proof}
The case $p=\infty$ in \eqref{665} is precisely the Bohnenblust--Hille
inequality, so we just need to consider $2m^{3}-4m^{2}+2m < p<\infty.$ Let
$\frac{2m-2}{m}\leq s\leq2$ and
\[
\lambda_{0,s}=\frac{2s}{ms+s+2-2m}.
\]
Note that
\begin{equation}
\label{invt2}
ms+s+2-2m>0 \qquad \text{and} \qquad 1\leq \lambda_{0,s}\leq 2.
\end{equation}

Since
\[
\frac{m-1}{s}+\frac{1}{\lambda_{0,s}}=\frac{m+1}{2},
\]
from the generalized Bohnenblust--Hille inequality (see \cite{alb}) we know
that there is a constant $C_{m}\geq1$ such that for all $m$-linear forms
$T:\ell_{\infty}^{n}\times\cdots\times\ell_{\infty}^{n}\rightarrow\mathbb{K}$
we have, for all $i=1,....,m,$%
\begin{equation}
\left(  \sum\limits_{j_{i}=1}^{n}\left(  \sum\limits_{\widehat{j_{i}}=1}
^{n}\left\vert T\left(  e_{j_{1}},...,e_{j_{m}}\right)  \right\vert
^{s}\right)  ^{\frac{1}{s}\lambda_{0,s}}\right)  ^{\frac{1}{\lambda_{0,s}}%
}\leq C_{m}\left\Vert T\right\Vert . \label{112}%
\end{equation}

Above, $\sum\limits_{\widehat{j_{i}}=1}^{n}$ means the sum over all $j_{k}$
for all $k\neq i.$ If we choose $s=\frac{2mp}{mp+p-2m}$ (note that this $s$ belongs to the interval $\left[\frac{2m-2}{m},2\right]$), we have $s>\frac{2m}{m+1}$ (this inequality is strict because we are considering the case $p<\infty$) and thus $\lambda_{0,s}<s$. In fact, $s>\frac{2m}{m+1}$ implies $ms+s>2m$ and thus adding $2$ in both sides of this inequality we can conclude that $$\frac{2s}{ms+s+2-2m}<\frac{2s}{2}=s,$$ i.e.,
\begin{equation}
\label{lambda0s}
\lambda_{0,s}<s.
\end{equation}

Since $p > 2m^{3}-4m^{2}+2m$ we conclude that
\[
s<\frac{2m^2-4m+2}{m^2-m-1}.
\]
Thus, from Theorem \ref{zmuzmu}, the optimal
constant associated to the multiple exponent
\[
\left(  \lambda_{0,s},s,s,...,s\right)
\]
is less than or equal to
\[
C_{m}=\prod\limits_{j=2}^{m}A_{\frac{2j-2}{j}}^{-1}.
\]
More precisely, \eqref{112} is valid with $C_{m}$ as above. Now the proof follows the same lines, \textit{mutatis mutandis}, of the proof of \cite[Theorem 1.1]{ap}.
\end{proof}

\begin{remark}
Note that it is simple to verify that these new estimates are better than the
old estimates. In fact, for complex scalars the inequality
\[
\prod\limits_{j=2}^{m}A_{\frac{2j-2}{j}}^{-1}<\left(  \frac{2}{\sqrt{\pi}%
}\right)  ^{\frac{2m\left(  m-1\right)  }{p}}\left(  \prod\limits_{j=2}%
^{m}A_{\frac{2j-2}{j}}^{-1}\right)  ^{\frac{p-2m}{p}}%
\]
is a straightforward consequence of
\[
\prod\limits_{j=2}^{m}A_{\frac{2j-2}{j}}^{-1}<\left(  \frac{2}{\sqrt{\pi}%
}\right)  ^{m-1},
\]
which is true for $m\geq3.$ The case of real scalars is analogous.
\end{remark}

Recall that from \cite{ap} we know that for $p\geq m^{2}$ the constants of the
Hardy--Littlewood inequality have a subpolynomial growth. The following graph
illustrates what we have thus far, combined with Theorem \ref{667}.

\begin{figure}[h]
\centering
\includegraphics[scale=0.7]{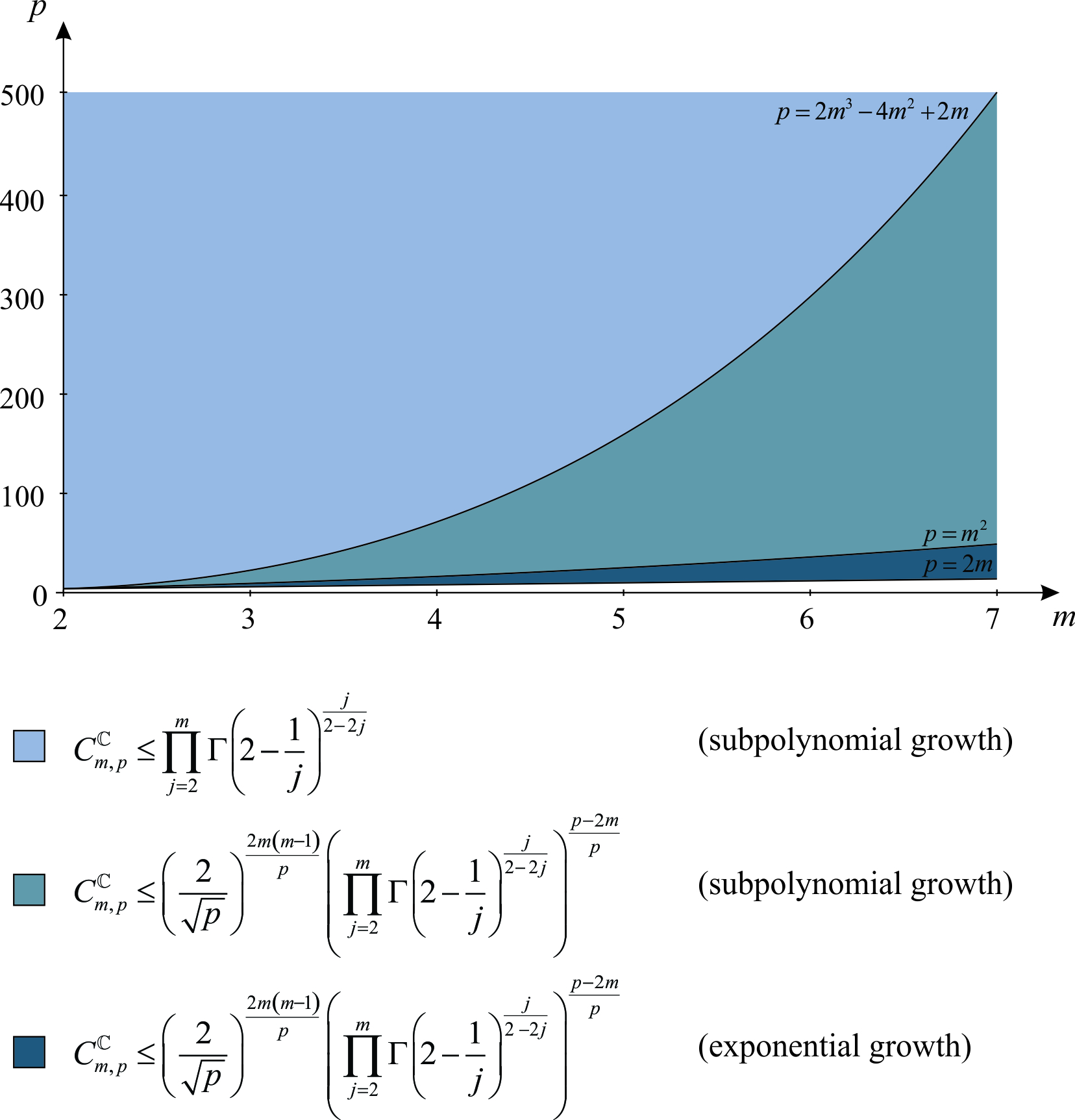}
\end{figure}

A question that arises naturally is: Are the optimal constants of the
Hardy--Littlewood and Bohnenblust--Hille inequalities the same? This result is
maybe slightly suggested by the above estimates. In addition, the best known
lower estimates for the real constants of the Hardy--Littlewood inequality (see
\cite{aplower}) are very similar to the respective lower estimates for the real constants
of the Bohnenblust--Hille inequality as it can be seen in \cite{diniz}. More
precisely, from \cite{aplower, diniz} we know that, for $m\geq2,$%
\[
C_{m,p}^{\mathbb{R}}>2^{\frac{mp+(6-4\log_{2}(1.74))m-2m^{2}-p}{mp}}>1
\]
and%
\[
\mathrm{B}_{\mathbb{R},m}^{\mathrm{mult}}\geq2^{1-\frac{1}{m}}\geq\sqrt{2}.
\]

\section{Application 2: Constants of the generalized Hardy-Littlewood inequality}

\label{sec4}

Given an integer $m\geq2$, the generalized Hardy--Littlewood inequality (see
\cite{alb, hardy,pra}) asserts that for $2m\leq p\leq\infty$ and
$\mathbf{q}:=(q_{1},...,q_{m})\in\left[  \frac{p}{p-m},2\right]  ^{m}$ such
that
\begin{equation}
\frac{1}{q_{1}}+...+\frac{1}{q_{m}}\leq\frac{mp+p-2m}{2p}, \label{ppoo}%
\end{equation}
there exists a constant $C_{m,p,\mathbf{q}}^{\mathbb{K}}\geq1$ such that, for
all continuous $m$--linear forms $T:\ell_{p}^{n}\times\cdots\times\ell_{p}%
^{n}\rightarrow\mathbb{K}$ and all positive integers $n$,
\[
\left(  \sum_{j_{1}=1}^{n}\left(  \sum_{j_{2}=1}^{n}\left(  \cdots\left(
\sum_{j_{m}=1}^{n}\left\vert T(e_{j_{1}},...,e_{j_{m}})\right\vert ^{q_{m}%
}\right)  ^{\frac{q_{m-1}}{q_{m}}}\cdots\right)  ^{\frac{q_{2}}{q_{3}}%
}\right)  ^{\frac{q_{1}}{q_{2}}}\right)  ^{\frac{1}{q_{1}}}\leq
C_{m,p,\mathbf{q}}^{\mathbb{K}}\left\Vert T\right\Vert .
\]
The best known estimates for the constants $C_{m,p,\mathbf{q}}^{\mathbb{K}}$
are $\left(  \sqrt{2}\right)  ^{m-1}$ for real scalars and $\left(  \frac
{2}{\sqrt{\pi}}\right)  ^{m-1}$ for complex scalars (see \cite{alb}). Very
recently, in \cite{ap} (and in the previous section, see \eqref{btt}), better
constants were obtained when $q_{1}=...=q_{m}=\frac{2mp}{mp+p-2m}.$ Now we
extend the results from \cite{ap} to general multiple exponents. Of course the
interesting case is the border case, i.e., when we have an equality in
\eqref{ppoo}. The proof is slightly more elaborated than the proof of Theorem
\ref{667} and also a bit more technical that the proof of the main result of \cite{ap}.

\begin{theorem}
Let $m\geq2$ be a positive integer and $2m<p\leq\infty$. Let also $\mathbf{q}:=\left(  q_{1},...,q_{m}\right)  \in\left[
\frac{p}{p-m},2\right]  ^{m}$ be such that
\[
\frac{1}{q_{1}}+...+\frac{1}{q_{m}}=\frac{mp+p-2m}{2p}.
\]

(i) If $\max q_{i}<\frac{2m^{2}-4m+2}{m^{2}-m-1}$, then
\[%
\begin{array}
[c]{llll}%
\displaystyle C_{m,p,\mathbf{q}}^{\mathbb{C}} & \leq & \prod\limits_{j=2}%
^{m}\Gamma\left(  2-\frac{1}{j}\right)  ^{\frac{j}{2-2j}},\vspace{0.2cm} & \\
C_{m,p,\mathbf{q}}^{\mathbb{R}} & \leq & \displaystyle 2^{\frac{446381}%
{55440}-\frac{m}{2}}\prod\limits_{j=14}^{m}\left(  \frac{\Gamma\left(
\frac{3}{2}-\frac{1}{j}\right)  }{\sqrt{\pi}}\right)  ^{\frac{j}{2-2j}}, &
\displaystyle\text{ if }m\geq14,\vspace{0.2cm}\\
C_{m,p,\mathbf{q}}^{\mathbb{R}} & \leq & \displaystyle\prod\limits_{j=2}%
^{m}2^{\frac{1}{2j-2}}, & \displaystyle\text{ if }2\leq m\leq13.
\end{array}
\]

(ii) If $\max q_{i}\geq\frac{2m^{2}-4m+2}{m^{2}-m-1}$, then
\[%
\begin{array}
[c]{llll}%
\displaystyle C_{m,p,\mathbf{q}}^{\mathbb{C}} & \leq & \left(  \frac{2}%
{\sqrt{\pi}}\right)  ^{2\left(  m-1\right)  \left(  \frac{m+1}{2}-\frac
{m}{\max q_{i}}\right)  }\left(  \displaystyle\prod\limits_{j=2}^{m}%
\Gamma\left(  2-\frac{1}{j}\right)  ^{\frac{j}{2-2j}}\right)  ^{m\left(
\frac{2}{\max q_{i}}-1\right)  },\vspace{0.2cm} & \\
C_{m,p,\mathbf{q}}^{\mathbb{R}} & \leq & \displaystyle 2^{\left(  m-1\right)
\left(  \frac{m+1}{2}-\frac{m}{\max q_{i}}\right)  }\left(  2^{\frac
{446381}{55440}-\frac{m}{2}}\prod\limits_{j=14}^{m}\left(  \frac{\Gamma\left(
\frac{3}{2}-\frac{1}{j}\right)  }{\sqrt{\pi}}\right)  ^{\frac{j}{2-2j}%
}\right)  ^{m\left(  \frac{2}{\max q_{i}}-1\right)  }, & \displaystyle\text{
if }m\geq14,\vspace{0.2cm}\\
C_{m,p,\mathbf{q}}^{\mathbb{R}} & \leq & \displaystyle 2^{\left(  m-1\right)
\left(  \frac{m+1}{2}-\frac{m}{\max q_{i}}\right)  }\left(  \prod
\limits_{j=2}^{m}2^{\frac{1}{2j-2}}\right)  ^{m\left(  \frac{2}{\max q_{i}%
}-1\right)  }, & \displaystyle\text{ if }2\leq m\leq13.
\end{array}
\]

\end{theorem}

\begin{proof}
Let us first suppose $\max q_{i}<\frac{2m^{2}-4m+2}{m^{2}-m-1}.$ The arguments
follow the general lines of \cite{ap}, but are slightly different and due the
technicalities we present the details for the sake of clarity. Define for
$s\in\left(  \max q_{i},\frac{2m^{2}-4m+2}{m^{2}-m-1}\right)  $,
\begin{equation}
\lambda_{m,s}=\frac{2ps}{mps+ps+2p-2mp-2ms}. \label{bvb}%
\end{equation}
Observe that $\lambda_{m,s}$ is well defined for all $s\in\left(\max q_i,\frac{2m^{2}-4m+2}{m^{2}-m-1}\right)$. In fact, as we have in (\ref{invt1})
 note that for all $s\in\left[\frac{2mp-2p}{mp-2m},2\right]$ we have $$mps+ps+2p-2mp-2ms>0 \qquad \text{and} \qquad \frac{p}{p-m}\leq \lambda_{m,s}\leq 2.$$ Since $s>\max q_{i}\geq\frac{2mp}{mp+p-2m}>\frac{2mp-2p}{mp-2m}$ (the last inequality is strict because we are not considering the case $p=2m$) and $\frac{2m^{2}-4m+2}{m^{2}-m-1}\leq 2$ it follows that $\lambda_{m,s}$ is well defined for all $s\in\left(\max q_i,\frac{2m^2-4m+2}{m^2-m-1}\right)$.

Let us prove
\begin{equation}
C_{m,p,\left(  \lambda_{m,s},s,...,s\right)  }^{\mathbb{K}}\leq\prod
\limits_{j=2}^{m}A_{\frac{2j-2}{j}}^{-1} \label{8866}%
\end{equation}
for all $s\in\left(  \max q_{i},\frac{2m^{2}-4m+2}{m^{2}-m-1}\right)$. In
fact, for these values of $s$, consider
\[
\lambda_{0,s}=\frac{2s}{ms+s+2-2m}.
\]
Observe that if $p=\infty$ then $\lambda_{m,s}=\lambda_{0,s}$. Since%
\[
\frac{m-1}{s}+\frac{1}{\lambda_{0,s}}=\frac{m+1}{2},
\]
from the generalized Bohnenblust--Hille inequality (see \cite{alb}) we know
that there is a constant $C_{m}\geq1$ such that for all $m$-linear forms
$T:\ell_{\infty}^{n}\times\cdots\times\ell_{\infty}^{n}\rightarrow\mathbb{K}$
we have, for all $i=1,....,m,$%

\begin{equation}
\left(  \sum\limits_{j_{i}=1}^{n}\left(  \sum\limits_{\widehat{j_{i}}=1}%
^{n}\left\vert T\left(  e_{j_{1}},...,e_{j_{m}}\right)  \right\vert
^{s}\right)  ^{\frac{1}{s}\lambda_{0,s}}\right)  ^{\frac{1}{\lambda_{0,s}}%
}\leq C_{m}\left\Vert T\right\Vert . \label{78}%
\end{equation}
Since
\[
\frac{2m}{m+1}\leq \frac{2mp}{mp+p-2m}\leq\max q_{i}<s < \frac{2m^{2}-4m+2}{m^{2}-m-1}%
\]
it is not to difficult to prove that (see \eqref{lambda0s})
\[
\lambda_{0,s}<s < \frac{2m^{2}-4m+2}{m^{2}-m-1}.
\]

Since $s < \frac{2m^{2}-4m+2}{m^{2}-m-1}$ we conclude by Theorem \ref{zmuzmu} that the optimal constant
associated to the multiple exponent%
\[
\left(  \lambda_{0,s},s,s,...,s\right)
\]
is less then or equal to
\begin{equation}
\prod\limits_{j=2}^{m}A_{\frac{2j-2}{j}}^{-1}. \label{9898}%
\end{equation}
More precisely, \eqref{78} is valid with $C_{m}$ as above.

Since $\lambda_{m,s}=\lambda_{0,s}$ if $p=\infty,$ we have \eqref{8866} for
all for all $s\in\left(  \max q_{i},\frac{2m^{2}-4m+2}{m^{2}-m-1}\right)$
and the proof is done for this case.

For $2m<p<\infty$, let%
\[
\lambda_{j,s}=\frac{\lambda_{0,s}p}{p-\lambda_{0,s}j}%
\]
for all $j=1,....,m.$ Note that%
\[
\lambda_{m,s}=\frac{2ps}{mps+ps+2p-2mp-2ms}%
\]
and this notation is compatible with \eqref{bvb}. Since $s>\max q_{i}\geq
\frac{2mp}{mp+p-2m}\geq\frac{2mp}{mp+p-2j}$ for all $j=1,...,m$ we also
observe that
\begin{equation}
\lambda_{j,s}<s \label{ult}%
\end{equation}
for all $j=1,....,m$. Moreover, observe that%
\[
\left(  \frac{p}{\lambda_{j,s}}\right)  ^{\ast}=\frac{\lambda_{j+1,s}}%
{\lambda_{j,s}}%
\]
for all $j=0,...,m-1$. Here, as usual, $\left(  \frac{p}{\lambda_{j,s}}\right)  ^{\ast}$
denotes the conjugate number of $\left(  \frac{p}{\lambda_{j,s}}\right)  $.
From now on part of the proof of (i) follows the steps of the proof of the main
result of \cite{ap}, but we prefer to show the details for the sake of completeness (note that the final part of the proof of (i) requires a more  subtle argument than the one used in  \cite{ap}).

Let us suppose that $1\leq k\leq m$ and that%
\[
\left(  \sum_{j_{i}=1}^{n}\left(  \sum_{\widehat{j_{i}}=1}^{n}\left\vert
T(e_{j_{1}},...,e_{j_{m}})\right\vert ^{s}\right)  ^{\frac{1}{s}%
\lambda_{k-1,s}}\right)  ^{\frac{1}{\lambda_{k-1,s}}}\leq C_{m}\Vert T\Vert
\]
is true for all continuous $m$--linear forms $T:\underbrace{\ell_{p}^{n}%
\times\cdots\times\ell_{p}^{n}}_{k-1\ \mathrm{times}}\times\ell_{\infty}%
^{n}\times\cdots\times\ell_{\infty}^{n}\rightarrow\mathbb{K}$ and for all
$i=1,...,m.$ Let us prove that%
\[
\left(  \sum_{j_{i}=1}^{n}\left(  \sum_{\widehat{j_{i}}=1}^{n}\left\vert
T(e_{j_{1}},...,e_{j_{m}})\right\vert ^{s}\right)  ^{\frac{1}{s}\lambda_{k,s}%
}\right)  ^{\frac{1}{\lambda_{k,s}}}\leq C_{m}\Vert T\Vert
\]
for all continuous $m$--linear forms $T:\underbrace{\ell_{p}^{n}\times
\cdots\times\ell_{p}^{n}}_{k\ \mathrm{times}}\times\ell_{\infty}^{n}%
\times\cdots\times\ell_{\infty}^{n}\rightarrow\mathbb{K}$ and for all
$i=1,...,m$.

The initial case (the case in which all $p=\infty$) is precisely \eqref{78}
with $C_{m}$ as in \eqref{9898}.

Consider%
\[
T\in\mathcal{L}(\underbrace{\ell_{p}^{n},...,\ell_{p}^{n}}_{k\ \mathrm{times}%
},\ell_{\infty}^{n},...,\ell_{\infty}^{n};\mathbb{R})
\]
and for each $x\in B_{\ell_{p}^{n}}$ define
\[%
\begin{array}
[c]{ccccl}%
T^{(x)} & : & \underbrace{\ell_{p}^{n}\times\cdots\times\ell_{p}^{n}%
}_{k-1\ \mathrm{times}}\times\ell_{\infty}^{n}\times\cdots\times\ell_{\infty
}^{n} & \rightarrow & \mathbb{R}\\
&  & (z^{(1)},...,z^{(m)}) & \mapsto & T(z^{(1)},...,z^{(k-1)},xz^{(k)}%
,z^{(k+1)},...,z^{(m)}),
\end{array}
\]
with $xz^{(k)}=(x_{j}z_{j}^{(k)})_{j=1}^{n}$. Observe that
\[
\Vert T\Vert=\sup\{\Vert T^{(x)}\Vert:x\in B_{\ell_{p}^{n}}\}.
\]
By applying the induction hypothesis to $T^{(x)}$, we obtain
\begin{equation}%
\begin{array}
[c]{l}%
\displaystyle\left(  \sum_{j_{i}=1}^{n}\left(  \sum_{\widehat{j_{i}}=1}%
^{n}\left\vert T\left(  e_{j_{1}},...,e_{j_{m}}\right)  \right\vert
^{s}\left\vert x_{j_{k}}\right\vert ^{s}\right)  ^{\frac{1}{s}\lambda_{k-1,s}%
}\right)  ^{\frac{1}{\lambda_{k-1,s}}}\vspace{0.2cm}\\
=\displaystyle\left(  \sum_{j_{i}=1}^{n}\left(  \sum_{\widehat{j_{i}}=1}%
^{n}\left\vert T\left(  e_{j_{1}},...,e_{j_{k-1}},xe_{j_{k}},e_{j_{k+1}%
},...,e_{j_{m}}\right)  \right\vert ^{s}\right)  ^{\frac{1}{s}\lambda_{k-1,s}%
}\right)  ^{\frac{1}{\lambda_{k-1,s}}}\vspace{0.2cm}\\
=\displaystyle\left(  \sum_{j_{i}=1}^{n}\left(  \sum_{\widehat{j_{i}}=1}%
^{n}\left\vert T^{(x)}\left(  e_{j_{1}},...,e_{j_{m}}\right)  \right\vert
^{s}\right)  ^{\frac{1}{s}\lambda_{k-1,s}}\right)  ^{\frac{1}{\lambda_{k-1,s}%
}}\vspace{0.2cm}\\
\leq C_{m}\Vert T^{(x)}\Vert\vspace{0.2cm}\\
\leq C_{m}\Vert T\Vert
\end{array}
\label{guga030}%
\end{equation}
for all $i=1,...,m.$

We will analyze two cases:

\begin{itemize}
\item $i=k.$
\end{itemize}

Since
\[
\left(  \frac{p}{\lambda_{j-1,s}}\right)  ^{\ast}=\frac{\lambda_{j,s}}%
{\lambda_{j-1}}%
\]
for all $j=1,...,m$, we conclude that
\[%
\begin{array}
[c]{l}%
\displaystyle\left(  \sum_{j_{k}=1}^{n}\left(  \sum_{\widehat{j_{k}}=1}%
^{n}\left\vert T\left(  e_{j_{1}},...,e_{j_{m}}\right)  \right\vert
^{s}\right)  ^{\frac{1}{s}\lambda_{k,s}}\right)  ^{\frac{1}{\lambda_{k,s}}%
}\vspace{0.2cm}\\
\displaystyle=\displaystyle\left(  \sum_{j_{k}=1}^{n}\left(  \sum
_{\widehat{j_{k}}=1}^{n}\left\vert T\left(  e_{j_{1}},...,e_{j_{m}}\right)
\right\vert ^{s}\right)  ^{\frac{1}{s}\lambda_{k-1,s}\left(  \frac{p}%
{\lambda_{k-1,s}}\right)  ^{\ast}}\right)  ^{\frac{1}{\lambda_{k-1,s}}\frac
{1}{\left(  \frac{p}{\lambda_{k-1,s}}\right)  ^{\ast}}} \vspace{0.2cm}\\
\displaystyle=\left\Vert \left(  \left(  \sum_{\widehat{j_{k}}=1}%
^{n}\left\vert T\left(  e_{j_{1}},...,e_{j_{m}}\right)  \right\vert
^{s}\right)  ^{\frac{1}{s}\lambda_{k-1,s}}\right)  _{j_{k}=1}^{n}\right\Vert
_{\left(  \frac{p}{\lambda_{k-1,s}}\right)  ^{\ast}}^{\frac{1}{\lambda
_{k-1,s}}}\vspace{0.2cm}\\
\displaystyle=\left(  \sup_{y\in B_{\ell_{\frac{p}{\lambda_{k-1,s}}}^{n}}}%
\sum_{j_{k}=1}^{n}|y_{j_{k}}|\left(  \sum_{\widehat{j_{k}}=1}^{n}\left\vert
T\left(  e_{j_{1}},...,e_{j_{m}}\right)  \right\vert ^{s}\right)  ^{\frac
{1}{s}\lambda_{k-1,s}}\right)  ^{\frac{1}{\lambda_{k-1,s}}}\vspace{0.2cm}\\
\displaystyle=\left(  \sup_{x\in B_{\ell_{p}^{n}}}\sum_{j_{k}=1}^{n}|x_{j_{k}%
}|^{\lambda_{k-1,s}}\left(  \sum_{\widehat{j_{k}}=1}^{n}\left\vert T\left(
e_{j_{1}},...,e_{j_{m}}\right)  \right\vert ^{s}\right)  ^{\frac{1}{s}%
\lambda_{k-1,s}}\right)  ^{\frac{1}{\lambda_{k-1,s}}}\vspace{0.2cm}\\
\displaystyle=\sup_{x\in B_{\ell_{p}^{n}}}\left(  \sum_{j_{k}=1}^{n}\left(
\sum_{\widehat{j_{k}}=1}^{n}\left\vert T\left(  e_{j_{1}},...,e_{j_{m}%
}\right)  \right\vert ^{s}\left\vert x_{j_{k}}\right\vert ^{s}\right)
^{\frac{1}{s}\lambda_{k-1,s}}\right)  ^{\frac{1}{\lambda_{k-1,s}}}%
\vspace{0.2cm}\\
\displaystyle\leq C_{m}\Vert T\Vert.
\end{array}
\]
where the last inequality holds by \eqref{guga030}.

\begin{itemize}
\item $i\neq k.$
\end{itemize}

It is clear that $\lambda_{k-1,s}<\lambda_{k,s}$ for all $1\leq
k\leq m$. Since $\lambda_{k,s}<s$ for all $1\leq
k\leq m$ (see \eqref{ult}) we get $$\lambda_{k-1,s}<\lambda_{k,s}<s \text{ for all }1\leq
k\leq m.$$ Denoting, for $i=1,....,m,$
\[
S_{i}=\left(  \sum_{\widehat{j_{i}}=1}^{n}|T(e_{j_{1}},...,e_{j_{m}}%
)|^{s}\right)  ^{\frac{1}{s}}%
\]
we get
\[%
\begin{array}
[c]{l}%
\displaystyle\sum_{j_{i}=1}^{n}\left(  \sum_{\widehat{j_{i}}=1}^{n}%
|T(e_{j_{1}},...,e_{j_{m}})|^{s}\right)  ^{\frac{1}{s}\lambda_{k,s}}%
=\sum_{j_{i}=1}^{n}S_{i}^{\lambda_{k,s}}=\sum_{j_{i}=1}^{n}S_{i}%
^{\lambda_{k,s}-s}S_{i}^{s}\vspace{0.2cm}\\
\displaystyle=\sum_{j_{i}=1}^{n}\sum_{\widehat{j_{i}}=1}^{n}\frac{|T(e_{j_{1}%
},...,e_{j_{m}})|^{s}}{S_{i}^{s-\lambda_{k,s}}}=\sum_{j_{k}=1}^{n}%
\sum_{\widehat{j_{k}}=1}^{n}\frac{|T(e_{j_{1}},...,e_{j_{m}})|^{s}}%
{S_{i}^{s-\lambda_{k,s}}}\vspace{0.2cm}\\
\displaystyle=\sum_{j_{k}=1}^{n}\sum_{\widehat{j_{k}}=1}^{n}\frac{|T(e_{j_{1}%
},...,e_{j_{m}})|^{\frac{s(s-\lambda_{k,s})}{s-\lambda_{k-1,s}}}}%
{S_{i}^{s-\lambda_{k,s}}}|T(e_{j_{1}},...,e_{j_{m}})|^{\frac{s(\lambda
_{k,s}-\lambda_{k-1,s})}{s-\lambda_{k-1,s}}}.
\end{array}
\]
Therefore, using H\"{o}lder's inequality twice we obtain
\begin{equation}%
\begin{array}
[c]{l}%
\displaystyle\sum_{j_{i}=1}^{n}\left(  \sum_{\widehat{j_{i}}=1}^{n}%
|T(e_{j_{1}},...,e_{j_{m}})|^{s}\right)  ^{\frac{1}{s}\lambda_{k,s}}%
\vspace{0.2cm}\\
\displaystyle\leq\sum_{j_{k}=1}^{n}\left(  \sum_{\widehat{j_{k}}=1}^{n}%
\frac{|T(e_{j_{1}},...,e_{j_{m}})|^{s}}{S_{i}^{s-\lambda_{k-1,s}}}\right)
^{\frac{s-\lambda_{k,s}}{s-\lambda_{k-1,s}}}\left(  \sum_{\widehat{j_{k}}%
=1}^{n}|T(e_{j_{1}},...,e_{j_{m}})|^{s}\right)  ^{\frac{\lambda_{k,s}%
-\lambda_{k-1,s}}{s-\lambda_{k-1,s}}}\vspace{0.2cm}\\
\displaystyle\leq\left(  \sum_{j_{k}=1}^{n}\left(  \sum_{\widehat{j_{k}}%
=1}^{n}\frac{|T(e_{j_{1}},...,e_{j_{m}})|^{s}}{S_{i}^{s-\lambda_{k-1,s}}%
}\right)  ^{\frac{\lambda_{k,s}}{\lambda_{k-1,s}}}\right)  ^{\frac
{\lambda_{k-1,s}}{\lambda_{k,s}}\cdot\frac{s-\lambda_{k,s}}{s-\lambda_{k-1,s}%
}}\times\left(  \sum_{j_{k}=1}^{n} \left(  \sum_{\widehat{j_{k}}=1}%
^{n}|T(e_{j_{1}},...,e_{j_{m}})|^{s}\right)  ^{\frac{1}{s}\lambda_{k,s}%
}\right)  ^{\frac{1}{\lambda_{k,s}}\cdot\frac{(\lambda_{k,s}-\lambda
_{k-1,s})s}{s-\lambda_{k-1,s}}}.
\end{array}
\label{huhu}%
\end{equation}
We know from the case $i=k$ that
\begin{equation}
\displaystyle\left(  \sum_{j_{k}=1}^{n}\left(  \sum_{\widehat{j_{k}}=1}%
^{n}|T(e_{j_{1}},...,e_{j_{m}})|^{s}\right)  ^{\frac{1}{s}\lambda_{k,s}%
}\right)  ^{\frac{1}{\lambda_{k,s}}\cdot\frac{(\lambda_{k,s}-\lambda
_{k-1,s})s}{s-\lambda_{k-1,s}}}\leq\left(  C_{m}\Vert T\Vert\right)
^{\frac{(\lambda_{k,s}-\lambda_{k-1,s})s}{s-\lambda_{k-1,s}}}. \label{huhi}%
\end{equation}
Now we investigate the first factor in \eqref{huhu}. From H\"{o}lder's
inequality and \eqref{guga030} it follows that
\begin{equation}%
\begin{array}
[c]{l}%
\displaystyle\left(  \sum_{j_{k}=1}^{n}\left(  \sum_{\widehat{j_{k}}=1}%
^{n}\frac{|T(e_{j_{1}},...,e_{j_{m}})|^{s}}{S_{i}^{s-\lambda_{k-1,s}}}\right)
^{\frac{\lambda_{k,s}}{\lambda_{k-1,s}}}\right)  ^{\frac{\lambda_{k-1,s}%
}{\lambda_{k,s}}}=\left\Vert \left(  \sum_{\widehat{j_{k}}}\frac{|T(e_{j_{1}%
},...,e_{j_{m}})|^{s}}{S_{i}^{s-\lambda_{k-1,s}}}\right)  _{j_{k}=1}%
^{n}\right\Vert _{\left(  \frac{p}{\lambda_{k-1,s}}\right)  ^{\ast}}%
\vspace{0.2cm}\\
\displaystyle=\sup_{y\in B_{\ell_{\frac{p}{\lambda_{k-1,s}}}^{n}}}\sum
_{j_{k}=1}^{n}|y_{j_{k}}|\sum_{\widehat{j_{k}}=1}^{n}\frac{|T(e_{j_{1}%
},...,e_{j_{m}})|^{s}}{S_{i}^{s-\lambda_{k-1,s}}}=\sup_{x\in B_{\ell_{p}^{n}}%
}\sum_{j_{k}=1}^{n}\sum_{\widehat{j_{k}}=1}^{n}\frac{|T(e_{j_{1}}%
,...,e_{j_{m}})|^{s}}{S_{i}^{s-\lambda_{k-1,s}}}|x_{j_{k}}|^{\lambda_{k-1,s}%
}\vspace{0.2cm}\\
\displaystyle=\sup_{x\in B_{\ell_{p}^{n}}}\sum_{j_{i}=1}^{n}\sum
_{\widehat{j_{i}}=1}^{n}\frac{|T(e_{j_{1}},...,e_{j_{m}})|^{s-\lambda_{k-1,s}%
}}{S_{i}^{s-\lambda_{k-1,s}}}|T(e_{j_{1}},...,e_{j_{m}})|^{\lambda_{k-1,s}%
}|x_{j_{k}}|^{\lambda_{k-1,s}}\vspace{0.2cm}\\
\displaystyle\leq\sup_{x\in B_{\ell_{p}^{n}}}\sum_{j_{i}=1}^{n}\left(
\sum_{\widehat{j_{i}}=1}^{n}\frac{|T(e_{j_{1}},...,e_{j_{m}})|^{s}}{S_{i}^{s}%
}\right)  ^{\frac{s-\lambda_{k-1,s}}{s}}\left(  \sum_{\widehat{j_{i}}=1}%
^{n}|T(e_{j_{1}},...,e_{j_{m}})|^{s}|x_{j_{k}}|^{s}\right)  ^{\frac{1}%
{s}\lambda_{k-1,s}}\vspace{0.2cm}\\
\displaystyle=\sup_{x\in B_{\ell_{p}^{n}}}\sum_{j_{i}=1}^{n}\left(
\sum_{\widehat{j_{i}}=1}^{n}|T(e_{j_{1}},...,e_{j_{m}})|^{s}|x_{j_{k}}%
|^{s}\right)  ^{\frac{1}{s}\lambda_{k-1,s}}\leq\left(  C_{m}\Vert
T\Vert\right)  ^{\lambda_{k-1,s}}.
\end{array}
\label{huho}%
\end{equation}
Replacing \eqref{huhi} and \eqref{huho} in \eqref{huhu} we conclude that
\begin{align*}
\displaystyle\sum_{j_{i}=1}^{n}\left(  \sum_{\widehat{j_{i}}=1}^{n}%
|T(e_{j_{1}},...,e_{j_{m}})|^{s}\right)  ^{\frac{1}{s}\lambda_{k,s}}  &
\leq\left(  C_{m}\Vert T\Vert\right)  ^{\lambda_{k-1,s}\frac{s-\lambda_{k,s}%
}{s-\lambda_{k-1,s}}}\left(  C_{m}\Vert T\Vert\right)  ^{\frac{(\lambda
_{k,s}-\lambda_{k-1,s})s}{s-\lambda_{k-1,s}}}\\
&  =\left(  C_{m}\Vert T\Vert\right)  ^{\lambda_{k,s}}%
\end{align*}
and finally the proof of \eqref{8866} is done for all $s\in\left(  \max
q_{i},\frac{2m^{2}-4m+2}{m^{2}-m-1}\right)$.

Now the proof uses a different argument from those from \cite{ap}, since a new
interpolation procedure is now needed. From (\ref{ult}) we know that $\lambda
_{m,s}<s$ for all $s\in\left(  \max q_{i},\frac{2m^{2}-4m+2}{m^{2}%
-m-1}\right)$. Therefore, using the Minkowski inequality as in \cite{alb}, it is
possible to obtain from \eqref{8866} that, for all fixed $i\in\left\{
1,...,m\right\}$,
\begin{equation}
C_{m,p,\left(  s,...,s,\lambda_{m,s},s,...,s\right)  }^{\mathbb{K}}\leq
\prod\limits_{j=2}^{m}A_{\frac{2j-2}{j}}^{-1}%
\end{equation}
for all $s\in\left(  \max q_{i},\frac{2m^{2}-4m+2}{m^{2}-m-1}\right)$ with
$\lambda_{m,s}$ in the $i$--th position. Finally, from Lemma \ref{lema} we
know that $\left(  q_{1}^{-1},...,q_{m}^{-1}\right)  $ belongs to the convex
hull of
\[
\left\{  \left(  \lambda_{m,s}^{-1},s^{-1},...,s^{-1}\right)  ,...,\left(
s^{-1},...,s^{-1},\lambda_{m,s}^{-1}\right)  \right\}
\]
for all $s\in\left(  \max q_{i},\frac{2m^{2}-4m+2}{m^{2}-m-1}\right)$ with
certain constants $\theta_{1,s},...,\theta_{m,s}$ and thus, from the
interpolative technique from \cite{alb}, we get
\begin{align*}
C_{m,p,\mathbf{q}}^{\mathbb{K}}  &  \leq\left(  C_{m,p,\left(  \lambda
_{m,s},s,...,s\right)  }^{\mathbb{K}}\right)  ^{\theta_{1,s}}\cdots\left(
C_{m,p,\left(  s,...,s,\lambda_{m,s}\right)  }^{\mathbb{K}}\right)
^{\theta_{m,s}}\\
&  \leq\left(  \prod\limits_{j=2}^{m}A_{\frac{2j-2}{j}}^{-1}\right)
^{\theta_{1,s}+\cdots+\theta_{m,s}}\\
&  =\prod\limits_{j=2}^{m}A_{\frac{2j-2}{j}}^{-1}.
\end{align*}

Now we prove (ii), which is simpler.

Define
\[
s_{\mathbf{q}}=\max q_{i}
\]
and, for $s\in\left[  \frac{2mp}{mp+p-2m},2\right]  $,%
\[
\lambda_{0,s}=\frac{2s}{ms+s+2-2m}%
\]
and
\begin{equation}
\lambda_{m,s}=\frac{2ps}{mps+ps+2p-2mp-2ms}. \label{bvbvvvvvvvv}%
\end{equation}
Since $\left[\frac{2mp}{mp+p-2m},2\right]\subseteq\left[\frac{2mp-2p}{mp-2m},2\right]$ we have $\frac{p}{p-m}\leq \lambda_{m,s}\leq 2$ (see (\ref{invt1})) and since $\left[\frac{2mp}{mp+p-2m},2\right]\subseteq\left[\frac{2m-2}{m},2\right]$ from \eqref{invt2} we know $1\leq \lambda_{0,s}\leq 2$.

We prove the case of real scalars. For complex scalars the proof is analogous,
and we can replace $\sqrt{2}$ by $\frac{2}{\sqrt{\pi}}$ and $\mathbb{R}$ by
$\mathbb{C}$. Since
\[
\frac{m-1}{s}+\frac{1}{\lambda_{0,s}}=\frac{m+1}{2},
\]
from the generalized Bohnenblust--Hille inequality (see \cite{alb}) we know
that there is a constant $C_{m}\geq1$ such that for all $m$-linear forms
$T:\ell_{\infty}^{n}\times\cdots\times\ell_{\infty}^{n}\rightarrow\mathbb{K}$
we have, for all $i=1,....,m,$%
\begin{equation}
\left(  \sum\limits_{j_{i}=1}^{n}\left(  \sum\limits_{\widehat{j_{i}}=1}%
^{n}\left\vert T\left(  e_{j_{1}},...,e_{j_{m}}\right)  \right\vert
^{s}\right)  ^{\frac{1}{s}\lambda_{0,s}}\right)  ^{\frac{1}{\lambda_{0,s}}%
}\leq C_{m}\left\Vert T\right\Vert . \label{7844}%
\end{equation}
Since
\[
\frac{2m}{m+1}\leq\frac{2mp}{mp+p-2m}\leq s\leq2,
\]
we know that
\begin{equation}\label{estrela}
\lambda_{0,s}\leq s\leq2.
\end{equation}
To verify the first inequality in \eqref{estrela} we just need to repeat the argument used to prove \eqref{lambda0s}, now supposing $s\geq\frac{2m}{m+1}$.

The multiple exponent%
\[
\left(  \lambda_{0,s},s,s,...,s\right)
\]
can be obtained by interpolating the multiple exponents $\left(
1,2...,2\right)  $ and $\left(  \frac{2m}{m+1},...,\frac{2m}{m+1}\right)  $
with, respectively,
\[%
\begin{array}
[c]{c}%
\displaystyle\theta_{1}=2\left(  \frac{1}{\lambda_{0,s}}-\frac{1}{s}\right)  ,
\vspace{0.2cm}\\
\displaystyle\theta_{2}=m\left(  \frac{2}{s}-1\right)  ,
\end{array}
\]
in the sense of \cite{alb}.

The exponent $\left(  \frac{2m}{m+1},...,\frac{2m}{m+1}\right)  $ is the
classical exponent of the Bohnenblust--Hille inequality and the estimate of
the constant associated to $\left(  1,2...,2\right)  $ is $\left(  \sqrt
{2}\right)  ^{m-1}$ (see, for instance, \cite{ap}, although this result is
very well-known)$.$

Therefore, the optimal constant associated to the multiple exponent%
\[
\left(  \lambda_{0,s},s,s,...,s\right)
\]
is less then or equal (for real scalars) to
\[
\left(  \left(  \sqrt{2}\right)  ^{m-1}\right)  ^{2\left(  \frac{1}%
{\lambda_{0,s}}-\frac{1}{s}\right)  }\left(  B_{\mathbb{R},m}^{\mathrm{mult}%
}\right)  ^{m\left(  \frac{2}{s}-1\right)  }%
\]
i.e.,%
\[
C_{m}\leq\left(  \sqrt{2}\right)  ^{2(m-1)\left(  \frac{m+1}{2}-\frac{m}%
{s}\right)  }\left(  B_{\mathbb{R},m}^{\mathrm{mult}}\right)  ^{m\left(
\frac{2}{s}-1\right)  }.
\]
More precisely, \eqref{7844} is valid with $C_{m}$ as above. For complex
scalars we can use the Khinchine inequality for Steinhaus variables and
replace $\sqrt{2}$ by $\frac{2}{\sqrt{\pi}}$ as in \cite{ddss}. Therefore,
analogously to the previous case (see also \cite[Theorem 1.1]{ap}), it is
possible to prove that
\begin{equation}
C_{m,p,(\lambda_{m,s},s,...,s)}^{\mathbb{R}}\leq\left(  \sqrt{2}\right)
^{2(m-1)\left(  \frac{m+1}{2}-\frac{m}{s}\right)  }\left(  B_{\mathbb{R}%
,m}^{\mathrm{mult}}\right)  ^{m\left(  \frac{2}{s}-1\right)  }
\label{lqlqlql000}%
\end{equation}
for all $s\in\left[  \frac{2mp}{mp+p-2m},2\right]  $.

Since $s\geq\frac{2mp}{mp+p-2m}$ we have $\lambda_{m,s}\leq s$ (in fact, we just need to imitate the argument used to prove \eqref{lambdams}, now supposing $s\geq\frac{2mp}{mp+p-2m}$) and so from
(\ref{lqlqlql000}), using the Minkowski inequality as in \cite{alb}, it is
possible to obtain, for all fixed $j\in\left\{  1,...,m\right\}  ,$%
\begin{equation}
C_{m,p,\left(  s,...,s,\lambda_{m,s},s,...,s\right)  }^{\mathbb{R}}\leq\left(
\sqrt{2}\right)  ^{2(m-1)\left(  \frac{m+1}{2}-\frac{m}{s}\right)  }\left(
B_{\mathbb{R},m}^{\mathrm{mult}}\right)  ^{m\left(  \frac{2}{s}-1\right)  }
\label{lqlqlql}%
\end{equation}
for all $s\in\left[  \frac{2mp}{mp+p-2m},2\right]  $ with $\lambda_{m,s}$ in
the $j$--th position. Therefore, given $\epsilon>0$ (sufficiently small),
consider%
\[
s_{\mathbf{q}+\epsilon}:=s_{\mathbf{q}}+\epsilon=\max q_{i}+\epsilon,
\]
and since $s_{\mathbf{q}+\epsilon}>\frac{2mp}{mp+p-2m}$ $\left(  \text{because
}s_{\mathbf{q}}=\max q_{i}\geq\frac{2mp}{mp+p-2m}\right)  $ we have
\eqref{lqlqlql} for $s=s_{\mathbf{q}+\epsilon}$. Finally, from Lemma
\ref{lema} we know that $\left(  q_{1}^{-1},...,q_{m}^{-1}\right)  $ belongs
to the convex hull of
\[
\left\{  \left(  \lambda_{m,s_{\mathbf{q}+\epsilon}}^{-1},s_{\mathbf{q}%
+\epsilon}^{-1},...,s_{\mathbf{q}+\epsilon}^{-1}\right)  ,...,\left(
s_{\mathbf{q}+\epsilon}^{-1},...,s_{\mathbf{q}+\epsilon}^{-1},\lambda
_{m,s_{\mathbf{q}+\epsilon}}^{-1}\right)  \right\}
\]
with certain constants $\theta_{1,s_{\mathbf{q}+\epsilon}},...,\theta
_{m,s_{\mathbf{q}+\epsilon}}$ and thus, from the interpolative technique from
\cite{alb}, we get
\[%
\begin{array}
[c]{l}%
C_{m,p,\mathbf{q}}^{\mathbb{R}}\vspace{0.2cm}\\
\leq\left(  C_{m,p,\left(  \lambda_{m,s_{\mathbf{q}+\epsilon}},s_{\mathbf{q}%
+\epsilon},...,s_{\mathbf{q}+\epsilon}\right)  }^{\mathbb{R}}\right)
^{\theta_{1,s_{\mathbf{q}+\epsilon}}}\cdots\left(  C_{m,p,\left(
s_{\mathbf{q}+\epsilon},...,s_{\mathbf{q}+\epsilon},\lambda_{m,s_{\mathbf{q}%
+\epsilon}}\right)  }^{\mathbb{R}}\right)  ^{\theta_{m,s_{\mathbf{q}+\epsilon
}}}\vspace{0.2cm}\\
\leq\left(  \left(  \sqrt{2}\right)  ^{2(m-1)\left(  \frac{m+1}{2}-\frac
{m}{s_{\mathbf{q}+\epsilon}}\right)  }\left(  B_{\mathbb{R},m}^{\mathrm{mult}%
}\right)  ^{m\left(  \frac{2}{s_{\mathbf{q}+\epsilon}}-1\right)  }\right)
^{\theta_{1,s_{\mathbf{q}+\epsilon}}+\cdots+\theta_{m,s_{\mathbf{q}+\epsilon}%
}}\\
=\left(  \sqrt{2}\right)  ^{2(m-1)\left(  \frac{m+1}{2}-\frac{m}%
{s_{\mathbf{q}+\epsilon}}\right)  }\left(  B_{\mathbb{R},m}^{\mathrm{mult}%
}\right)  ^{m\left(  \frac{2}{s_{\mathbf{q}+\epsilon}}-1\right)  }%
\end{array}
\]
for all $\epsilon>0$ sufficiently small. By making $\epsilon\rightarrow0$ we
get the result.
\end{proof}

\end{document}